\documentclass[letterpaper, 10pt,conference]{ieeeconf}
\usepackage{blindtext}
\usepackage{graphicx}
\IEEEoverridecommandlockouts   % \overrideIEEEmargins
\usepackage{amssymb,amsmath,amsfonts,color}
\usepackage{amsmath,amsfonts}
\usepackage{tikz,cite}
\usepackage{epstopdf} 
\usepackage{psfrag}
\usepackage{latexsym}
\usepackage{epstopdf}
\usepackage{xspace}

\usetikzlibrary{decorations.pathmorphing}

\newtheorem{theorem}{Theorem}

\newtheorem{proposition}{Proposition}
\newtheorem{lemma}{Lemma}
\newtheorem{remark}{Remark}
\newtheorem{definition}{Definition}
\newtheorem{example}{Example}

\usepackage{tcolorbox}
\usepackage{algorithm}

\begin{document}
\title{\LARGE \bf   Cooperative Vehicle Speed Fault Diagnostics and Correction}

\author{Mohammad Pirani, Ehsan Hashemi, Amir Khajepour, Baris Fidan, Bakhtiar Litkouhi and Shih-Ken Chen  
\thanks{ M. Pirani, E. Hashemi,  A. Khajepour and B. Fidan   are with the Department of Mechanical and Mechatronics Engineering, University of Waterloo, ON, N2L 3G1, Canada  (email: mpirani, ehashemi, a.khajepour, fidan @uwaterloo.ca).
Bakhtiar Litkouhi and Shih-Ken Chen  are with the R\&D Department, General Motors Co., Warren, MI 48093, USA}}

\maketitle

\thispagestyle{empty}
\pagestyle{empty}

%%%%%%%%%%%%%%%%%%%%%%%%%%%%%%%%%%%%%%%%%%%%%%%%%%%%%%%%%%%%%%

\begin{abstract}
Reliable estimation (or measurement) of vehicle states has always been an active topic of research in the automotive industry and academia. Among the vehicle states, vehicle speed has a priority due to its critical importance in traction and stability control. Moreover, the emergence of new generation of communication technologies has brought a new avenue to traditional studies on vehicle estimation and control. To this end, this paper introduces a set of distributed function calculation algorithms for vehicle networks, robust to communication failures. The  introduced algorithms enable each vehicle to gather information from other vehicles in the network in a distributed manner. A procedure to use such a bank of information for a single vehicle to diagnose and correct a possible fault in its own speed estimation/measurement is discussed. The functionality and performance of the proposed algorithms are verified via illustrative examples and simulation results. 

\end{abstract} 

%\begin{IEEEkeywords}
%Distributed calculation, Cooperative reliability assessment, Vehicle networks, Speed fault diagnostics
%\end{IEEEkeywords}

%%%%%%%%%%%%%%%%%%%%%%%%%%%%%%%%%%%%%%%%%%%%%%%%%%%%%%
\section{Introduction} \label{sec:Intro}
%%%%%%%%%%%%%%%%%%%%%%%%%%%%%%%%%%%%%%%%%%%%%%%%%%%%%%
The emergence of {\it Internet of vehicles}, as a tangible representation of Internet of things, has significantly changed the shape of the urban transportation \cite{Shensherman, Linshen, Stankovic}. The rate of growth of this field of research has become so high that it is expected that the traffic of Internet of vehicles duplicates in next few years. There are various applications of Internet of vehicles which are primarily pertaining the driving safety issues. The advent of ever growing vehicle-to-vehicle (V2V) communications enable drivers to transmit and receive necessary information which can be potentially used for traffic management, fuel consumption issues, as well as increasing the reliability of individual vehicle state estimation and control \cite{noureddin, Tubaishat, Patra}. More particularly, vehicle active safety systems can take advantage of the extra information obtained from the network of connected vehicles to design more reliable vehicle velocity estimation algorithms, which undergoes significant research during the past decade \cite{Wenzel,Antonov2011a, piraniTCST, HashemiCEP}.

Unfortunately, there are  few efficient distributed calculation algorithms\footnote{These algorithms are sometimes referred to as {\it distributed estimation algorithms}; however, in this paper we use distributed calculation as used in \cite{Sundaram2011} in order to avoid confusion with single vehicle state estimation.}  for vehicle ad-hoc networks, compatible with the current available inter-vehicle communication standards. To this end, this paper applies a distributed  calculation algorithm to obtain vehicles' states (position, speed and acceleration) from other vehicles in the network and proposes a procedure to diagnose and correct faults (if any) in the speed of any single vehicle.

Distributed calculation techniques have been widely investigated during the past decade \cite{Olfatishamma, Olfatifilter,Riberio,  Kamgarpour, Fawzi}. The essence of distributed function calculation is that each member of the network calculates some quantity from all other members in the network (even from the members which are not directly connected to) via only local interactions (communicating with neighbors)\cite{Sundaram2011, Sundaram, Varayaa, Cattivelli}. More concisely, some global information is obtained via using local interactions \cite{Shreyasmitra, Speranzon}.  One of the main contributions of distributed calculation algorithms is to translate the dynamics (system-theoretic) notions of observability into network-theoretic concepts. In this direction, this paper uses one of the developed distributed calculation algorithms which is compatible with vehicle ad-hoc networks to estimate vehicle speeds in a distributed manner. It also proposes an algorithm for utilizing such  globally gathered network information to diagnose possible failures in speed measurement (or estimation) of each single vehicle.

The contributions of this paper are: 
\begin{enumerate}
    \item We apply one of the developed distributed calculation algorithms in \cite{Sundaram2011} to a network of connected vehicles which is robust to communication failures. Based on the sufficient condition proposed for the algorithm, we introduce a class of vehicle networks which certifies a specific level of connectivity. The proposed distributed calculation methodology benefits reliable vehicle state estimation, which is essential in related researches in automotive applications. 
    
    \item We discuss on the ways to use the information gathered from the distributed calculation module to diagnose and correct if there exists a fault in the measured (or estimated) speed of any single vehicle in the network. 
\end{enumerate}
 The paper is organized as follows. In Section \ref{sec:not} some notations and mathematical definitions are introduced and a brief overview of  inter-vehicular communications standards is discussed. Section \ref{sec:ffwsgf} defines the problem of distributed calculation of vehicle states (position, speed and acceleration) in a given network to diagnose faults in the speed of each vehicle. In Section \ref{sec:Sec3} the distributed calculation algorithm is introduced for both cases where there is no communication failure and when there exist some failures in inter-vehicular communications. Vehicle speed fault diagnosis and correction algorithms, based on the gathered information from the network, are discussed in Section \ref{sec:fdetcor}. Section \ref{sec:Results} presents some illustrative examples, which show the effectiveness of the discussed algorithms. Section \ref{sec:Conclusion} concludes the paper.

\section{Notation and Definitions} \label{sec:not}

This section introduces some mathematical notations used in this paper, as well as a brief introduction to vehicle-to-vehicle (V2V) communications.

\subsection{Notation}
In this paper, an undirected network is denoted by  $\mathcal{G}=\{\mathcal{V},\mathcal{E}\}$,  where $\mathcal{V} = \{v_1, v_2, \ldots, v_n\}$ is a set of nodes (or vertices) and $\mathcal{E} \subset \mathcal{V}\times\mathcal{V}$ is the set of edges.  Neighbors of node $v_i \in \mathcal{V}$ are given by the set $\mathcal{N}_i = \{v_j \in \mathcal{V} \mid (v_i, v_j) \in \mathcal{E}\}$. The  degree of each node $v_i$ is denoted by $d_i=|\mathcal{N}_i|$. A $k$-nearest neighbor platoon, $\mathcal{P}(n,k)$, is a network comprised of $n$ nodes (here vehicles),  where each node can communicate with its $k$ nearest neighbors from its back and $k$ nearest neighbors from its front, for some $k \in \mathbb{N}$. This definition is compatible with vehicle networks,  due to the limited  sensing and communication range for each  vehicle and the distance between the consecutive vehicles \cite{PiraniITS}. An example of $\mathcal{P}(n,k)$ is shown in Fig. \ref{fig:2nndef}. This specific network topology, other than its applicability in analyzing vehicle-to-vehicle communication scenarios, has a particular network property, called $k$-connectivity, which will be used later in Section \ref{sec:Results}.
%-----------------------------------------------------
\begin{figure}[h!]
\centering
\includegraphics[width=1\linewidth]{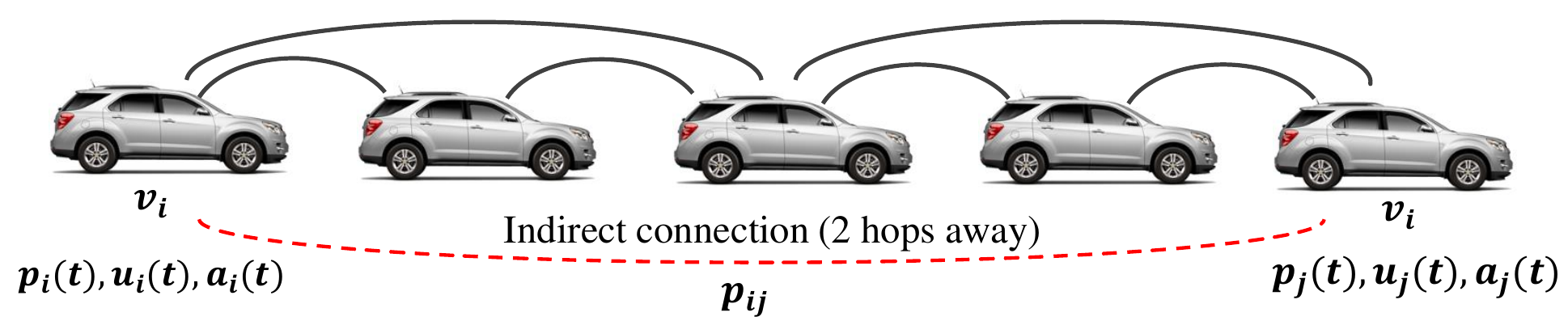}
\caption{A $2$-nearest neighbor platoon of 5 vehicles, $\mathcal{P}(5,2)$.}
\label{fig:2nndef}
\end{figure}
%-----------------------------------------------------

\subsection{V2V Communications}

Dedicated Short-Range Communications (DSRC) are one-way or two-way short-range to medium-range wireless communication channels that are specifically designed for automotive applications. It works in 5.9 GHz band, uses 75 MHz bandwidth channel, and provides low latency communication for safety applications. However, existing data sources and communications media (e.g. mobile phone and radio frequency) could be used for non-safety applications as well. The range of coverage of such communication is up to 1000 meters and the rate of data transmission is up to 27 Mbps. On of the standards which supports some DSRC applications is SAE J2735. There are various information that are sent via this standard, including (but not limited to): temporal ID of the vehicle, time of transmission, latitude, longitude and elevation of the vehicle, longitudinal speed and acceleration of the vehicle, and break system status. Among various information supported by this standard, in this paper, the followings are used: (i) vehicle $v_i$'s position (called $p_i$), (ii) its longitudinal velocity (speed), $u_i$, and (iii) longitudinal acceleration, $a_i$.

\section{Problem definition}\label{sec:ffwsgf}

Consider a network of connected vehicles $v_1, v_2, .., v_n$ which is represented via a graph $\mathcal{G}=\{\mathcal{V}, \mathcal{E}\}$. These vehicles can communicate with each other and share or disseminate information throughout the network. As mentioned in Section \ref{sec:not}, we only focus on the case where vehicles send their  positions, speeds (longitudinal velocities) and longitudinal accelerations to each other.  Each vehicle $v_i$ is able to estimate (or measure) its longitudinal velocity, ${u}_i$, via a specific method.\footnote{The velocity estimation method of each single vehicle is out of the scope of this paper and the reader is encouraged to see  \cite{Wenzel,Antonov2011a, piraniTCST, HashemiCEP}.} The goal for vehicle $v_i$ is to find out a measure for assessing the correctness of its speed $u_i$.  Vehicle $v_i$ should utilize an independent source of information. Such an independent source of information can be the values (positions, velocities and accelerations) of other vehicles, together with an algorithm to relate these obtained values to each other.  Therefore, a distributed calculation algorithm is introduced which enables vehicle $v_i$ to use  in order to gather the positions, speeds and accelerations of other vehicles in the network and then assess the reliability of its own  speed based on the gathered data. The overall procedure is:
 
\begin{enumerate}
    \item[(1)] \textbf{Distributed Calculation:} Each vehicle $v_i$ applies a distributed algorithm to gather the information of all of the vehicles in the network via using only local information exchange (with the vehicles in its communication range, called neighbors). Two different versions of this algorithm are proposed in this paper. In the first algorithm, it is assumed that there is no failure in vehicle communications and in the second algorithm, we assume that some of the vehicles fail to disseminate their information properly.

    \item[(2)] \textbf{Speed Fault Diagnostics:} After obtaining the information from other vehicles, vehicle $v_i$ performs an algorithm, based on (relative) positions, velocities and accelerations, to find out whether its own velocity is faulty or correct.

    \item[(3)] \textbf{Speed Correction:} If vehicle $v_i$ finds out that its self-speed-measurement is faulty (based on the previous step), it uses the information of other vehicles in the network to correct its own velocity. In particular, $v_i$ calculates its speed from the eye of other vehicles in the network and performs an averaging to increase the reliability of its speed. 
    
\end{enumerate}

An overview of the above three steps is schematically shown in Fig. \ref{fig:2nndddef}.
%-----------------------------------------------------
\begin{figure}[h!]
\centering
\includegraphics[width=0.8\linewidth]{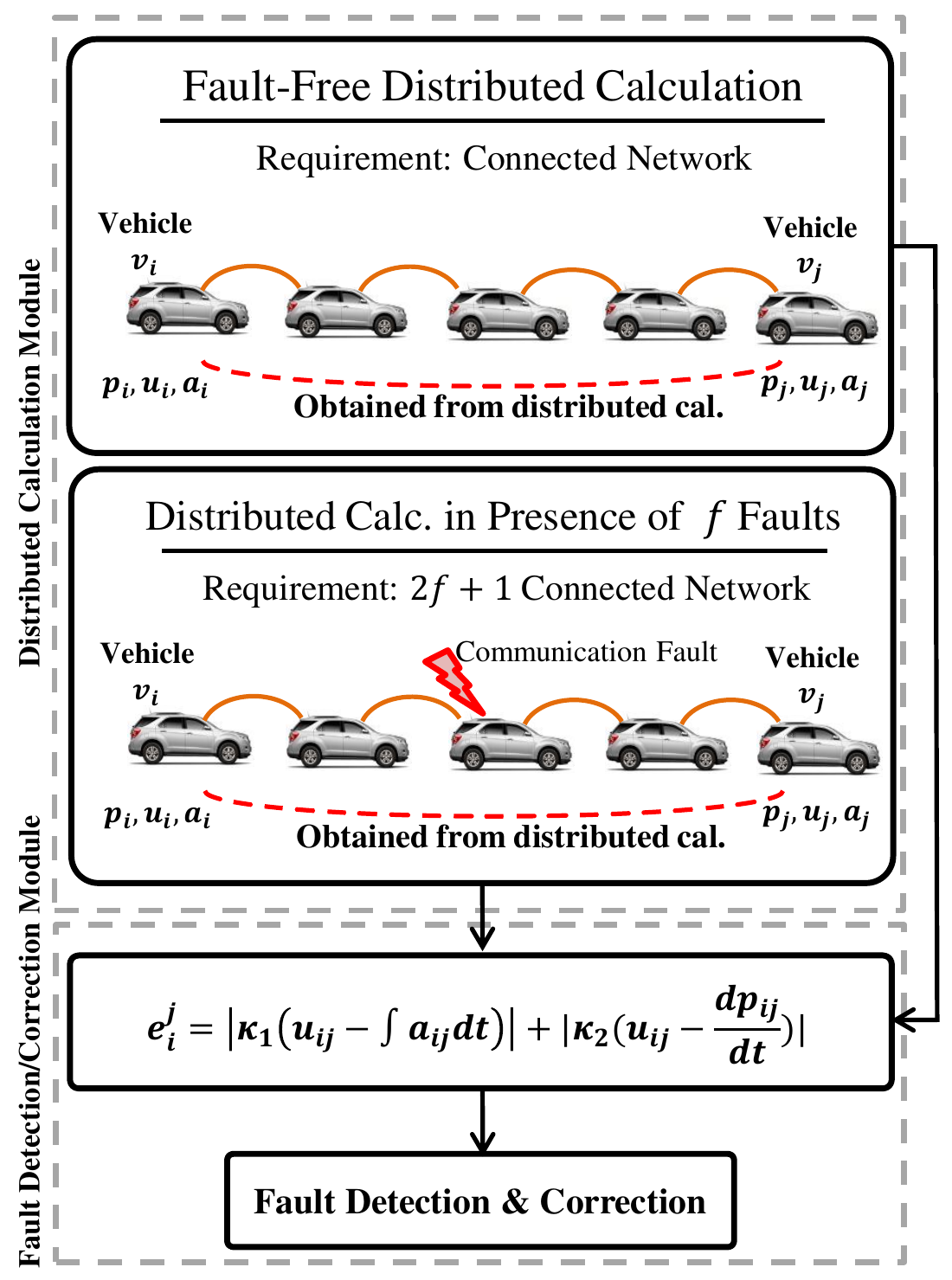}
\caption{Distributed velocity reliability assessment procedure.}
\label{fig:2nndddef}
\end{figure}
%-----------------------------------------------------
In the following section, we begin by introducing distributed calculation algorithm (step (1)). The fault diagnosis and speed correction schemes (steps 2 and 3) will be discussed in Section \ref{sec:fdetcor}. 

%%%%%%%%%%%%%%%%%%%%%%%%%%%%%%%%%%%%%%%%%%%%%%%%%%%%%%
\section{Robust Distributed Calculation} \label{sec:Sec3}
%%%%%%%%%%%%%%%%%%%%%%%%%%%%%%%%%%%%%%%%%%%%%%%%%%%%%%
In this section we introduce an  algorithm for a vehicle to gather the velocities of all vehicles in a network in a distributed manner, for two cases: (i) communication fault free case, and (ii) communication fault prone case. The theoretical foundation of these algorithms are derived in \cite{Sundaram2011, Sundaram}. The state of vehicle $v_j$, which can be its position $p_j$, speed $u_j$ or acceleration $a_j$ in this paper, is denoted simply by scalar $x_j[0]$. The objective is that this value becomes available for  vehicle $v_i$ in the network which is not in the communication range of vehicle $v_j$, as shown in Fig. \ref{fig:2nndef}. To yield this, vehicle $v_i$ performers a linear iteration policy using the following  time invariant updating rule
\begin{equation}
x_i[k+1]=w_{ii}x_i[k]+\sum_{j\in \mathcal{N}_i}w_{ij}x_j[k],
\label{eqn:1}
\end{equation}
where $w_{ii},w_{ij}>0$ are some predefined weights. Dynamics \eqref{eqn:1} in the vector form can be written as
\begin{equation}
\mathbf{x}[k+1]=\mathcal{W} \hspace{0.5 mm} \mathbf{x}[k],
\label{eqn:2}
\end{equation}
where $\mathcal{W}_{n\times n}$ is a matrix which shows the communication between vehicles in the network. In addition to dynamics \eqref{eqn:1}, at each time step, vehicle $v_i$ has access to its own value and the values of its neighbors. Hence, the output measurements for $v_i$ is defined as
\begin{equation}
y_i[k]=\mathcal{C}_i\mathbf{x}[k],
\label{eqn:measu}
\end{equation}
where $\mathcal{C}_i$ is a $(d_i +1)\times n$ matrix with a single 1 in each row that denotes the positions of the state-vector $\mathbf{x}[k]$ available to vehicle $v_i$ (i.e., these positions correspond to vehicles that are neighbors of $v_i$, along with vehicle  $v_i$ itself). The aim of such information dissemination is for each vehicle to obtain the initial condition of all other vehicles after running the linear dynamics \eqref{eqn:2} after some time steps.  

\begin{remark}
It should be noted that the evolution of state $\mathbf{x}[k]$ in \eqref{eqn:2} does not mean the evolution of vehicle's position, speed or acceleration, and dynamics \eqref{eqn:2} is only used for implementing a distributed calculation algorithm. Here, it is only $\mathbf{x}[0]$ that reflects the physical states of the vehicles. Moreover, in this study, we consider the fact that the communication between vehicles, dynamics \eqref{eqn:2}, is much faster than vehicle motions. In particular, although calculated quantities (positions, speeds and accelerations) obtained from the distributed  algorithm  are available after few time steps in inter-vehicle communications, such a time delay is negligible due to the sluggishness of vehicle's physical motion compared to the inter-vehicular data transmission rate. 
\end{remark}

Now suppose that there exist some vehicles that do not obey \eqref{eqn:2} to update their value.  More formally, at time step $k$, vehicle $v_i$'s attitude deviates from the predefined policy \eqref{eqn:1} and adds an arbitrary value, $\zeta_i[k]$, to its updating policy\footnote{In the literature such agents are called {\it adversarial} or {\it malicious} agents.}. In this case, the updating rule \eqref{eqn:1} will turn to 
\begin{equation}
x_i[k+1]=w_{ii}x_i[k]+\sum_{j\in \mathcal{N}_i}w_{ij}x_j[k]+\zeta_i[k],
\label{eqn:1wd}
\end{equation}
and if there are $f$ of these faulty vehicles, \eqref{eqn:1wd} in vector form  becomes 
\begin{equation}
\mathbf{x}[k+1]=\mathcal{W} \mathbf{x}[k]+\underbrace{[\mathbf{e}_1 \hspace{3mm} \mathbf{e}_2 \hspace{3mm} ... \hspace{3mm} \mathbf{e}_f]}_{\mathcal{B}}\boldsymbol{\zeta}[k],
\label{eqn:f2}
\end{equation}
where $\boldsymbol{\zeta}[k]=[\zeta_1[k], \zeta_2[k], ..., \zeta_f[k]]^T$ and $\mathbf{e}_j$  denotes an $n\times 1$ unit vector with a single nonzero entry with value 1 at its $j$-th position. From this view, dynamics \eqref{eqn:2} is a special case  of \eqref{eqn:f2} when there is no faulty vehicle. The continuous-time version of the above network dynamics is studied extensively \cite{PiraniSundaramArxiv, ArxiveRobutness}. The set of faulty vehicles in \eqref{eqn:f2} is unknown and consequently matrix $\mathcal{B}$ is  unknown.  The objective is for  vehicle $v_i$ to truly observe the initial values of all other vehicles, despite the actions of the faulty vehicles.  Based on  \eqref{eqn:f2}, the set of all values seen by vehicle $v_i$ during the first $L + 1$ time-steps of the linear iteration is given by
\begin{multline}
\underbrace{\begin{bmatrix}
       y_i[0]       \\[0.3em]
        y_i[1] \\[0.3em]
        \vdots \\[0.3em]
        y_i[L] \\[0.3em]
     \end{bmatrix}}_{y_i[0:L-1]}=\underbrace{\begin{bmatrix}
       \mathcal{C}_i      \\[0.3em]
       \mathcal{C}_i\mathcal{W} \\[0.3em]
        \vdots \\[0.3em]
        \mathcal{C}_i\mathcal{W}^L \\[0.3em]
     \end{bmatrix}}_{\mathcal{O}_{i,L}}\mathbf{x}[0]+\\
     \underbrace{\begin{bmatrix}
       0 & 0 & \cdots & 0      \\[0.3em]
       \mathcal{C}_i\mathcal{B} & 0 & \cdots & 0  \\[0.3em]
       \vdots & \vdots & \ddots & \vdots   \\[0.3em]
        \mathcal{C}_i\mathcal{W}^{L-1}\mathcal{B} &  \mathcal{C}_i\mathcal{W}^{L-2}\mathcal{B} & \cdots & \mathcal{C}_i\mathcal{B}      \\[0.3em]
     \end{bmatrix}}_{\mathcal{M}_{i,L}}\underbrace{\begin{bmatrix}
       \boldsymbol{\zeta}[0]       \\[0.3em]
        \boldsymbol{\zeta}[1] \\[0.3em]
        \vdots \\[0.3em]
        \boldsymbol{\zeta}[L] \\[0.3em]
     \end{bmatrix}}_{\boldsymbol{\zeta}[0:L-1]}
     \label{eqn:vesfw}
\end{multline} 
where $\mathcal{O}_{i,L}$ and $\mathcal{M}_{i,L}$ are called observability and invertability matrices, respectively.  In order to estimate the initial states, the network dynamical system \eqref{eqn:f2} together with the output measurement \eqref{eqn:measu} should satisfy certain observability conditions. Thus, we analyze the distributed state estimation for the fault-free case and then generalize this to the case where some vehicles update their states with some faults. 
\begin{remark}
The distributed calculation algorithm in the presence of vehicle communication fault analyzed in this paper contains the scenario which a vehicle stops receiving signal from its neighbors. This is the well-known notion called {\it  signal packet drop} which is studied in the communication literature \cite{Hajitouri, Sieler, Hespsurvey}. More formally, in \eqref{eqn:1wd} if we set $\zeta_i[k]= -\sum_{j\in \mathcal{N}_i}w_{ij}x_j[k]$, it becomes equivalent to the case where $v_i$ does not receive the data from its neighbors. Since the analysis in this paper does not depend on the value of $\zeta_i[k]$, the packet dropping scenario can be straightforwardly included in the robust distributed calculation analysis. 
\end{remark}

In the following subsection, we discuss on observability of dynamics \eqref{eqn:2} and output equation \eqref{eqn:measu} when there is no fault in inter-vehicle communications.

\subsection{Fault-Free Case}

For the case where there is no fault in inter-vehicular communications, the second term in \eqref{eqn:vesfw} (containing $\boldsymbol{\zeta}[0:L-1]$) does not exist. In this case, for each vehicle $v_i$ to be able to observe the initial conditions of other vehicles via its measurements, the system should be simply observable, i.e., its observability matrix $\mathcal{O}_{i,L}$ should be full row rank. More formally, if vehicle $v_i$ wants to observe the initial condition of vehicle $v_j$ in the network after $L$ time steps, the row space of $\mathcal{O}_{i,L}$ should contain $\mathbf{e}_j$. The following theorem introduces a freedom in designing the weight matrix $\mathcal{W}$ such that the observability matrix is guaranteed to be full row rank. 

\begin{theorem}[\cite{Sundaram}]
Let $\mathcal{G}(\mathcal{V},\mathcal{E})$ be a fixed and connected graph and define 
$$\mathcal{S}_i=\{v_j| \text{There exists a path from $v_j$ to $v_i$ in $\mathcal{G}$} \}\cup v_i.$$
Moreover, let $\epsilon_i$ be the distance of the farthest vehicle in the network to $v_i$. 
Then for almost\footnote{The {\it{almost}} in Theorem \ref{thm:leb} is due to the fact that  the set of  parameters for which the system is not observable has Lebesgue measure zero \cite{Reinschke}.} any choice of weight matrix $\mathcal{W}$, vehicle $v_i$ can obtain the initial value of vehicle $v_j\in \mathcal{S}_i$ after running the linear iteration \eqref{eqn:2} and output equation \eqref{eqn:measu} after at least $\epsilon_i$ and at most $|\mathcal{S}_i|-d_i$ time steps. 
\label{thm:leb}
\end{theorem}
 In order to implement the observer algorithm, each vehicle must have an access to the observability matrix to calculate the initial conditions. More precisely, in order to find the initial condition of vehicle $v_j$, $\mathbf{e}_j$ should be in the row space of the observability matrix. Hence, in order to find the initial condition of all vehicles, $I_{(n-1)\times (n-1)}$ should be in the row space of the observability matrix $\mathcal{O}_{i,L}$. If this condition is satisfied, vehicle $v_i$ can find a matrix like $\Gamma_i$ such that
\begin{equation}
\Gamma_i\mathcal{O}_{i,L_i}=I_{(n-1)\times (n-1)},
\label{eqn:gammadarbiar}
\end{equation}
and  based on \eqref{eqn:vesfw} the vector of initial conditions will be obtained. Hence, each vehicle in the network should calculate the observability matrix distributedly to be able to calculate matrix $\Gamma_i$. This is discussed in the following subsection.

\subsection{Distributed Calculation of the Observability Matrix}\label{sec:discsla}
Here we introduce a distributed algorithm to calculate the observability matrix, based on \cite{Sundaram}. Once the set of weights for matrix $\mathcal{W}$ in \eqref{eqn:2} is chosen, each vehicle  performs $n$ runs of the linear iteration, each for $n-d_{min}$ time-steps. For the $j$--th run, vehicle $v_j$ sets its initial condition
to be 1, i.e. $\bar{\mathbf{x}}_j[0]=\mathbf{e}_j$, and all other nodes set their initial conditions to be zero. After performing all runs, each vehicle $v_i$ has access to the following matrix. 
\begin{equation}
\Psi_i= \begin{bmatrix}
       y_{i,1}[0]  &   y_{i,2}[0] & \cdots & y_{i,n}[0] \\[0.3em]
        y_{i,1}[1] &  y_{i,2}[1] & \cdots & y_{i,n}[1]\\[0.3em]
        \vdots & \vdots &   \ddots &  \vdots\\[0.3em]
        y_{i,1}[k_i] &   y_{i,2}[k_i] & \cdots & y_{i,n}[k_i]\\[0.3em]
     \end{bmatrix},
\end{equation}
where $k_i=n-d_i-1$. Using \eqref{eqn:vesfw} $\Psi_i$ can be written as 
\begin{equation}
\Psi_i= \mathcal{O}_{i,n-d_i-1} [\bar{x}_1[0], \bar{x}_2[0], ..., \bar{x}_n[0]],
\label{eqn:obsdist}
\end{equation}
and since  $\bar{\mathbf{x}}_i[0]=\mathbf{e}_i$ we conclude that $\Psi_i=\mathcal{O}_{i,n-d_i-1}$. Hence, each vehicle has access to its observability matrix in a distributed way. From this, each vehicle is able to calculate matrix $\Gamma_i$ from \eqref{eqn:gammadarbiar}.

\subsection{Fault-Prone Case}

In this subsection, we extend what was discussed in the previous subsection to the case where there are some vehicles which fail to communicate properly. Such failures result in misleading other vehicles to run the distributed calculation algorithm appropriately. In this case, the values measured  by vehicle $v_i$ is given by \eqref{eqn:vesfw} where $\boldsymbol{\zeta}[0:L-1]$ is no longer zero for faulty vehicles. Hence, one should design an algorithm to find  initial conditions despite of the actions of these vehicles. It clearly demands a stronger condition on the system observability, as discussed in the following theorem.
\begin{theorem}[\cite{Sundaram2011}]
Suppose that there exists an integer $L$ and a weight matrix $\mathcal{W}$ such that, for all possible sets of faulty vehicles $\mathcal{I}$ of $2f$ faulty vehicles, the matrices $\mathcal{O}_{i,L}$ and $\mathcal{M}_{i,L}^{\mathcal{I}}$ for vehicle $v_i$ satisfy 
\begin{equation}
{\rm rank} \left([\mathcal{O}_{i,L} \quad \mathcal{M}_{i,L}^{\mathcal{I}}  ]\right)=n+ {\rm rank} \left( \mathcal{M}_{i,L}^{\mathcal{I}}\right).
\end{equation}
Then, if the nodes run the linear iteration for $L+1$ time steps with the weight matrix $\mathcal{W}$, vehicle $v_i$ can calculate initial values $x_1[0], x_2[0], ..., x_n[0]$, even if when up to $f$ vehicles fail to update their states correctly. 
\label{thm:conditionnn}
\end{theorem}

The proof of the above theorem provides a procedure for a vehicle in the network, like $v_i$, to recover the initial conditions of all vehicles in the network, as described bellow. 
\begin{itemize}
    \item First $v_i$ should find a matrix $\mathcal{N}_{i,L}$ whose rows form a basis for the left null space of the invertability matrix $\mathcal{M}_{i,L}$. After finding such matrix, by left multiplying it to $\mathcal{M}_{i,L}\boldsymbol{\zeta}[0:L-1]$ in \eqref{eqn:vesfw}, the second term in that equation is eliminated and we have $\mathcal{N}_{i,L}y_i[0:L-1]=\mathcal{N}_{i,L}\mathcal{O}_{i,L}$. 
    
    \item The next step is to define 
    \begin{equation}
    \mathcal{P}_{i,L}=\left(\mathcal{N}_{i,L}\mathcal{O}_{i,L}\right)^{\dagger}\mathcal{N}_{i,L},
    \label{eqn:alg1}
    \end{equation}
    where $(.)^{\dagger}$ is the left inverse of a matrix. Then vehicle $v_i$ can recover the initial conditions based on its measurements after $L$ time steps, as follows 
    \begin{equation}
    \mathcal{P}_{i,L}y_i[0:L-1]=\mathbf{x}[0].
    \label{eqn:alg2}
    \end{equation}
\end{itemize}

Theorem \ref{thm:conditionnn} provides a procedure for reconstructing initial conditions; however, it does not provide a graph-theoretic condition for state estimation. For this, the following definition is provided. 
\begin{definition}
A vertex-cut in a graph $\mathcal{G}=\{\mathcal{V}, \mathcal{E}\}$ is a subset $\mathcal{S} \subset \mathcal{V}$ such that removing the vertices in $\mathcal{S}$ (and the associated edges) from the graph causes the remaining graph to be disconnected. More specifically, a $(j, i)$-cut in a graph is a subset $\mathcal{S}_{ij}\subset \mathcal{V}$ such that removing the vertices in $\mathcal{S}_{ij}$ (and the associated edges) from the graph causes the graph to have no paths from vertex $v_j$ to vertex $v_i$.  Let $\kappa_{ij}$ denote the size of the smallest $(j, i)$-cut between any two vertices $v_j$ and $v_i$. Then graph $\mathcal{G}$ is said to be $k$-connected if $\kappa_{ij}=k$.
\end{definition}

\begin{theorem}[\cite{Sundaram2011}]
Let  $\mathcal{G}(\mathcal{V},\mathcal{E})$ be a fixed graph and let $f$ denote the maximum number of faulty vehicles that are to be tolerated in the network. Then, regardless of the actions of the faulty vehicles, $v_i$ can uniquely determine all of the initial values in the network via a linear iterative strategy if $\mathcal{G}$ is at least  $2f +1$ connected.
\label{thm:robdeist}
\end{theorem}

In Section \ref{sec:Results} we in introduce $k$-nearest neighbor vehicle platoons as examples of $k$-connected graphs, which are compatible with the physics of vehicle communication networks.

\section{Fault Detection and Correction} \label{sec:fdetcor}

In the previous section, we discuss the distribution algorithms that each vehicle can perform to obtain the required data to be used for speed fault detection and reconstruction. This section pertains the latter steps, in which vehicle $v_i$ applies the data it gathered to find out if there exists a failure in its own speed measurement (or estimation).

\subsection{Fault Detection}

After finding the values for other vehicles, based on the overall procedure introduced in Section  \ref{sec:ffwsgf} and shown in Fig. \ref{fig:2nndddef}, vehicle $v_i$ does the following error computation:
    \begin{multline}
    e_i^j=\kappa_1\left(u_{ij}-\int a_{ij} dt \right) 
    +\kappa_2\left(u_{ij}-\frac{d(p_{ij})}{dt}  \right)
    \label{eqn:partial}
\end{multline}
where $\kappa_1, \kappa_2>0$ are some design constants and $p_{ij}=p_i(t)-p_j(t)$, $u_{ij}=u_i(t)-u_j(t)$ and $a_{ij}(t)=a_i(t)-a_j(t)$ are relative distance, velocity and acceleration of vehicles $v_i$ and $v_j$, which are calculated based on absolute values obtained from the distributed calculation algorithm discussed in the previous section. We should define an (adaptive) threshold value $e_{th}$ such that if $|e_i^j| > e_{th}$ then it refers to the existence of some fault. For each vehicle $v_i$, based on the error parameter in \eqref{eqn:partial} for all vehicles $v_j$ in the network, we propose the following decision rule:
\begin{tcolorbox}
If $|e_i^j|>e_{th}$ for only a specific vehicle $v_j \in \mathcal{V}$, then the speed $u_j$ is faulty. If $|e_i^j|>e_{th}$ for (almost) all $v_j \in \mathcal{V}$, then the speed $u_i$ is faulty.  
\end{tcolorbox}

\subsection{Speed Correction}

After diagnosing a fault in the speed measurement of vehicle $v_i$, the final step is to make a correction. This step also takes advantage of the information that $v_i$ has obtained from the rest of the vehicles in the network. More formally, $v_i$ calculates speed $u_i$ from one of the following relations\footnote{Since integration of sensory measurement is prone to drifting effects, because of unavoidable signals bias, it would be better to use the signal derivatives.}
\begin{align}
u_i^j(t)&=u_j(t)+\frac{d(p_{ij})}{dt}, \nonumber \\
u_i^j(t)&=u_j(t)+\int a_{ij} dt,
\label{eqn:correctionvi}
\end{align}
for each vehicle $v_j\in \mathcal{V}$, where $u_i^j(t)$ is the speed of vehicle $v_i$ from the eye of vehicle $v_j$. Here $u_i^j(t)$ is called the opinion of vehicle $v_j$ about the velocity $u_i$. After calculating the opinions of all vehicles about $u_i$, i.r., $u_i^j$ for $v_j\in \mathcal{V}$, vehicle $v_i$ can apply a majority voting rule \cite{Peleg} or use some sensor fusion techniques to increase the reliability of the obtained values. Such post process methods on received signals (opinions) are necessary since these signals are prone to noises. One of these techniques is a distributed consensus algorithm to reduce the error of the resulting quantity, as discussed in \cite{BoydLall}. However, distributed consensus algorithms should be run $n$ times ($n$ is the size of the network), which produces huge computational complexity for large vehicle networks. Instead, here we re-frame the algorithm in \cite{BoydLall} in the following form. In this case, instead of doing distributed consensus, vehicle $v_i$ does a local averaging from the opinions it has computed about its speed. More specifically, vehicle $v_i$ calculates $u_i^j$ for all $v_j\in \mathcal{V}$, and then does the averaging as 
\begin{equation}
\bar{u}_i=\frac{1}{n}\sum _{j=1, j\neq i}^n u_i^j,
\label{eqn:averagingg}
\end{equation}
where $\bar{u}_i$ is called the average opinion of vehicles in the network about speed $u_i$. At the end, $v_i$ replaces its current velocity $u_i$ with the resulting average opinion $\bar{u}_i$. 

\begin{remark}
In an ideal case, where all noisy opinions $u_i^j$ are random variables with the same mean $\mu$ and variance $\sigma^2$ and are independent and identically distributed, it is well-known that the variance of the average signal $\bar{u}_i$ scales with $\frac{1}{n}$, where $n$ is the size of the network. This shows the advantage of the resulting average opinion $\bar{u}_i$ compared to each individual opinion $u_i^j$. 
\end{remark}

The overall procedure of distributed calculation and fault detection and correction is summarized in Algorithm 1. 

%-----------------------------------------------------
\begin{algorithm}[t]
\footnotesize{\caption{\texttt{Distributed Speed Fault Diagnosis/Correction for Vehicle $v_i$.}}}
\label{alg:Opin_Dyn}
\texttt{// Inputs:} $p_i$, $u_i$, $a_i$ and $n$ (network size)\\
\texttt{\textbf{Distributed Observability Matrix Computation:}} \\
Vehicle $v_i$ runs \eqref{eqn:1wd} for the $i$-th row of (stable) $\mathcal{W}$, and for $n-d_{min}$ time steps to calculate the observability matrix, using \eqref{eqn:obsdist}.\\
\texttt{\textbf{Distributed Calculation:}} \\
Vehicle $v_i$ runs \eqref{eqn:1wd}  for at least $L=\epsilon_i$ and at most $L=|\mathcal{S}_i|-d_i$ time steps
(If $x_i[k]$ does not go to zero, then there is at least a faulty vehicle in the network). \\
Vehicle $v_i$ finds matrix $\mathcal{P}_{i,L}$ via  \eqref{eqn:alg1} and calculates initial conditions vector $\mathbf{x}[0]$ using \eqref{eqn:alg2}.\\
\texttt{\textbf{Fault Diagnosis:}}\\
Vehicle $v_i$ calculated \eqref{eqn:partial}. If $|e_i^j|>e_{th}$, for some $e_{th}$, for only a specific vehicle $v_j \in \mathcal{V}$, then speed $u_j$ is faulty. If $|e_i^j|>e_{th}$ for (almost) all vehicles $v_j \in \mathcal{V}$, then speed $u_i$ is faulty.,\\
\texttt{\textbf{Fault Correction:}}\\
Vehicle $v_i$ uses \eqref{eqn:correctionvi} for all $v_j\in \mathcal{V}\setminus \{v_i\}$  to correct its own speed and does the local averaging \eqref{eqn:averagingg} to reduce the signal variance.\\
\texttt{// Output:} Reliable  and corrected speed $\bar{u}_i$.\\
\end{algorithm}
%-----------------------------------------------------

%%%%%%%%%%%%%%%%%%%%%%%%%%%%%%%%%%%%%%%%%%%%%%%%%%%%%%
\section{Simulations and Numerical Analysis of Network Topology} \label{sec:Results}
%%%%%%%%%%%%%%%%%%%%%%%%%%%%%%%%%%%%%%%%%%%%%%%%%%%%%%
In this section, some illustrative examples are demonstrated to confirm and clarify  Theorems \ref{thm:leb}, \ref{thm:conditionnn} and \ref{thm:robdeist} as well as fault diagnosis/correction algorithms mentioned in the previous section. First, the following proposition is presented.

\begin{proposition}
A $k$-nearest neighbor platoon, $\mathcal{P}(n,k)$, is a $k$-connected graph.
\label{prop:kconnectedprop}
\end{proposition} 
\begin{proof}
It is clear that in a $k$-nearest neighbor platoon structure, there exists exactly $k$ vertex disjoint paths\footnote{Vertex disjoint paths in a graph are paths which do not have  vertex in common (other than the first and last vertices).} between each couple of vertices $v_i, v_j \in \mathcal{V}$. Hence, according to Menger's theorem \cite{Aharoni}, the  size of the minimum vertex cut between $v_i$ and $v_j$ in a $k$-nearest neighbor platoon is $k$, which proves the claim.
\end{proof}
Based on the above proposition,  we can use $\mathcal{P}(n,k)$ as a $k$- connected graph in the simulations. According to Proposition \ref{prop:kconnectedprop} and based on the conditions mentioned in Theorems  \ref{thm:leb} and \ref{thm:robdeist}, the following figure schematically shows the vehicle network connectivity requirements for (robust) distributed calculation algorithms.

%-----------------------------------------------------
\begin{figure}[h!]
\centering
\includegraphics[width=0.75\linewidth]{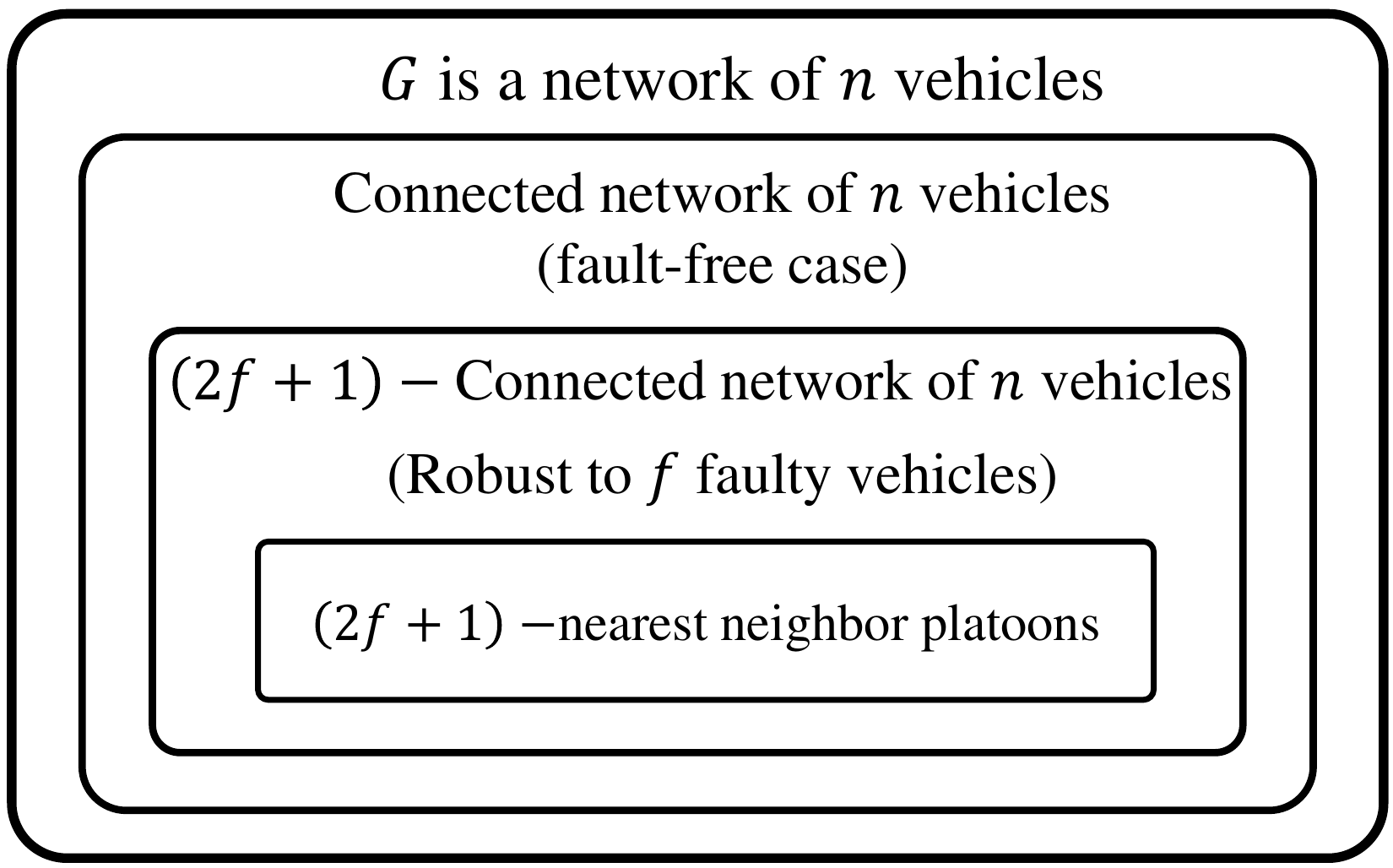}
\caption{Venn diagram of the vehicle network connectivity required for distributed calculation, based on Theorems  \ref{thm:leb}, \ref{thm:robdeist}.}
\label{fig:2nnddfuef}
\end{figure}
%-----------------------------------------------------

\subsection{Fault-Free Case}
For the case of fault free estimation, consider a 1 nearest-neighbour platoon (simple path graph) comprised of eight vehicles, Fig. \ref{fig:2nsdfndef} (top), in which there is no faulty vehicle in the network. 
%-----------------------------------------------------
\begin{figure}[h!]
\centering
\includegraphics[width=1\linewidth]{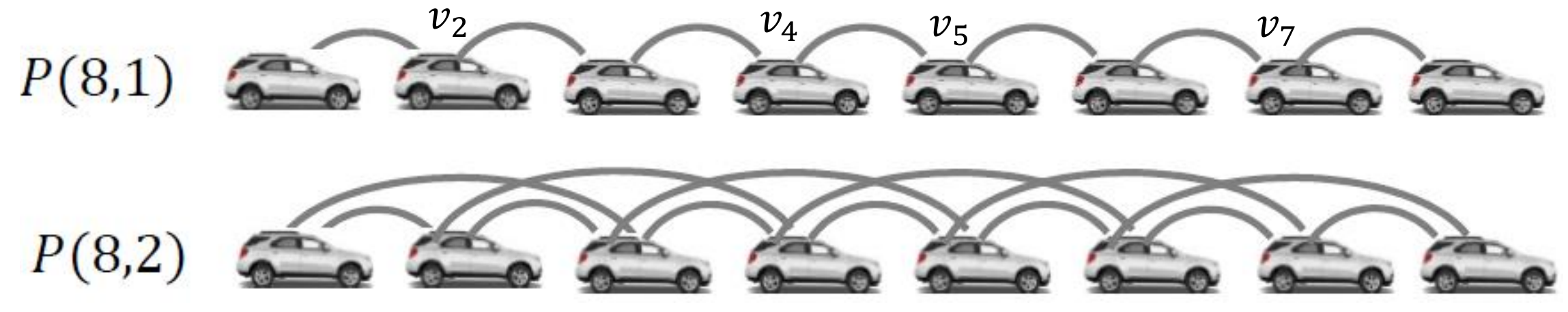}
\caption{Fault free $1$-nearest neighbor platoon of eight vehicles (top) and $2$-nearest neighbor platoon (bottom).}
\label{fig:2nsdfndef}
\end{figure}
%-----------------------------------------------------
The Euclidean norm of the error of the calculated quantities by vehicles 2, 4, 5, and 7 are depicted in Fig. \ref{fig:knn1}. According to this figure, all four vehicles estimate the states in finite time. However, vehicles 4 and 5 reached to the correct values in four time steps, while vehicles 2 and 7 reach in six time steps. These two values for time steps show that in this network topology, the lower bound for time steps mentioned in Theorem \ref{thm:leb} is achieved, which is equal to the longest distance of each of vehicles 2, 4, 5, and 7 to any other vehicle in the network. 
%-----------------------------------------------------
\begin{figure}[h!]
\centering
\includegraphics[width=0.95\linewidth]{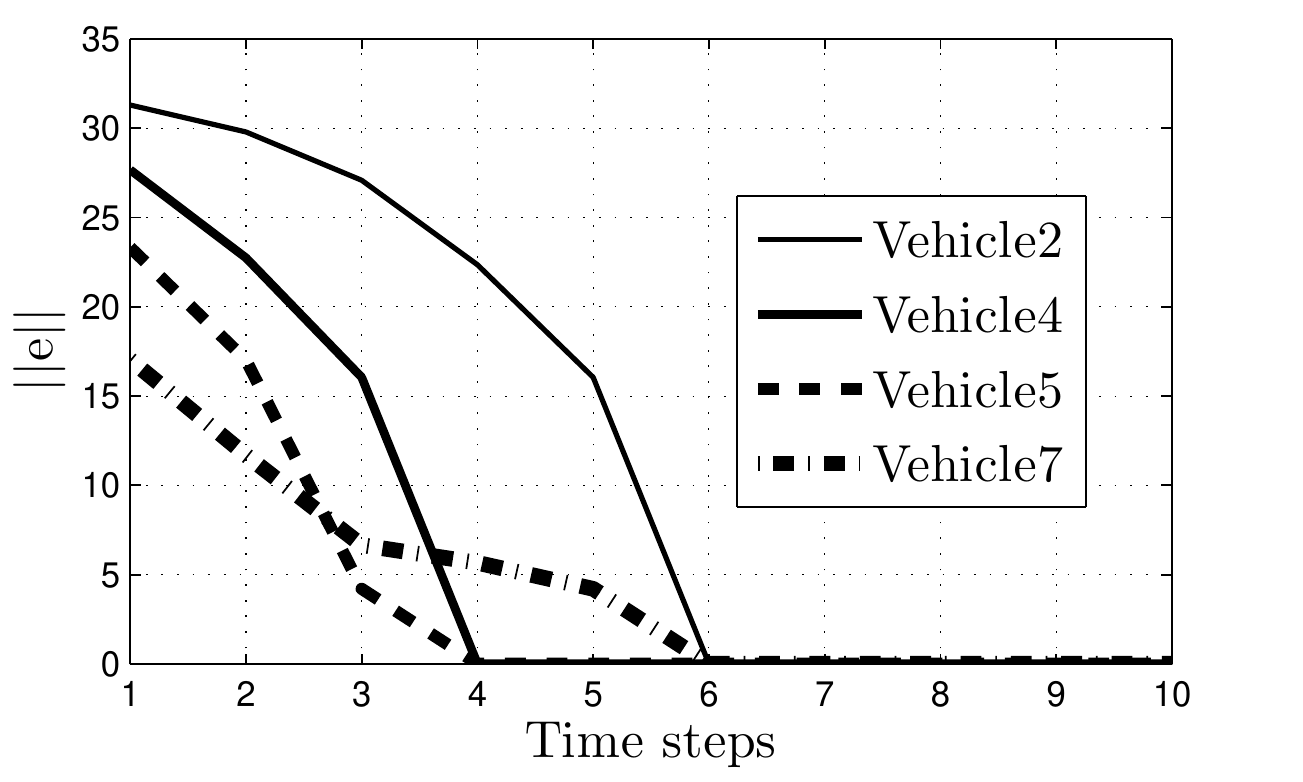}
\caption{Distributed calculation error for four vehicles in an 1 nearest-neighbour platoon, without fault.}
\label{fig:knn1}
\end{figure}
%-----------------------------------------------------
In Fig. \ref{fig:knn2} the connectivity of the network has increased by using a 2-nearest neighbor network of eight vehicles. In this case, vehicles 2 and 7 are two hops away from the heads of the network and vehicles 4 and 5 are only one hop away. Hence, increasing the connectivity of the network results in faster distributed calculation. 
%-----------------------------------------------------
\begin{figure}[h!]
\centering
\includegraphics[width=0.95\linewidth]{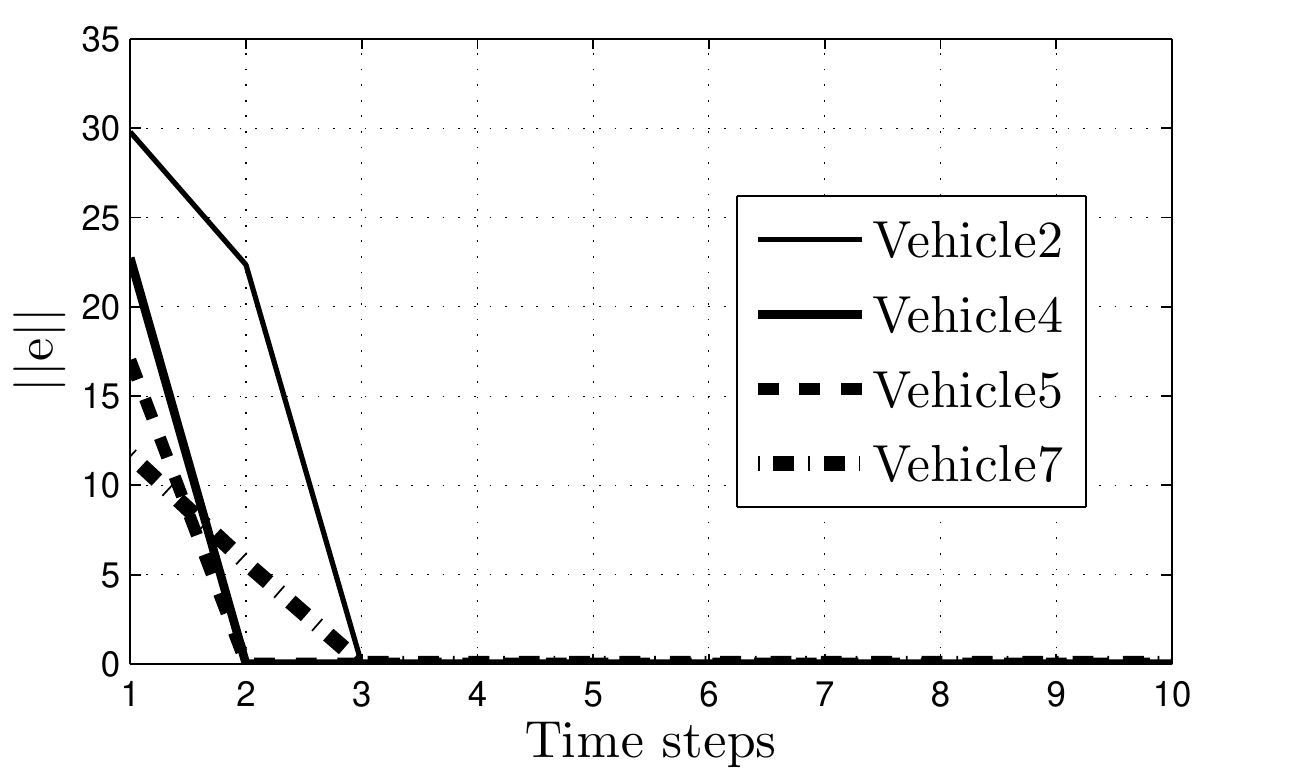}
\caption{Distributed calculation error for four vehicles in an 2 nearest-neighbour platoon, without fault.}
\label{fig:knn2}
\end{figure}
%-----------------------------------------------------

\subsection{Fault-Prone Case}

In another scenario, we assume that there is a vehicle which fails to update its state correctly in the distributed calculation setting. Based on Theorem \ref{thm:robdeist}, the network should satisfy a stronger connectivity condition in order to be able to tolerate the fault. In particular, the network should be at least 3-connected such that it guarantees to tolerate one faulty vehicle. For this, we choose a larger network, which is a 3-nearest neighbor platoon of 20 vehicles.  Fig. \ref{fig:knn3} shows the response of the system (Euclidean norm of states over time) for both fault-free (top) and faulty (bottom) cases. As matrix $\mathcal{W}$ used in \eqref{eqn:f2} here is chosen to be stable, having states go to some number other than zero means there is an extra input (fault) in the system.

%-----------------------------------------------------
\begin{figure}[h!]
\centering
\includegraphics[width=0.95\linewidth]{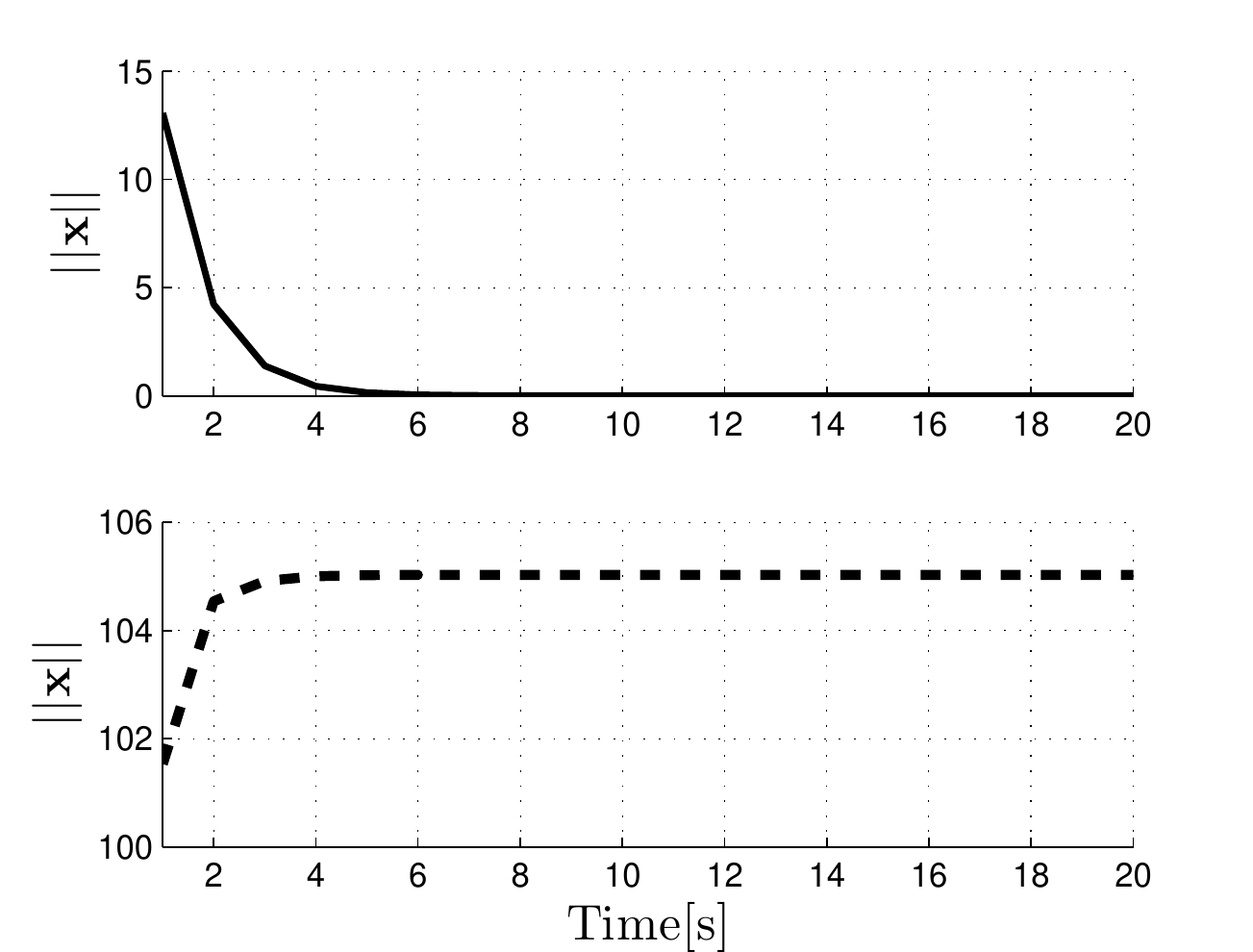}
\caption{Euclidean norm of states in dynamics \eqref{eqn:f2} on $\mathcal{P}(20,3)$, with no faulty vehicle (top), and with a faulty vehicle (bottom).}
\label{fig:knn3}
\end{figure}
%-----------------------------------------------------
Fig. \ref{fig:knn4} shows the  Euclidean norm of the error of the calculated quantities by one of the vehicles in the network $v_i$ in the case where one of the vehicles injects a faulty input to its updating rule. As it is shown in the figure, for the cases of 2 and 3- nearest neighbor platoons, $v_i$ manages to estimate the states, despite of the existence of a faulty vehicle in the network. It should be noted that although the 2-nearest neighbor platoon does not satisfy 3-connectivity condition, the estimation works properly. It is due to the fact that in this case, the location of the faulty vehicle and the faulty value it injects does not mislead $v_i$ to calculate the functions. As it is inferred form  Fig. \ref{fig:knn4}, a 1-nearest neighbor platoon can not tolerate a faulty vehicle in the network anymore. 

%-----------------------------------------------------
\begin{figure}[h!] 
\centering
\includegraphics[width=0.95\linewidth]{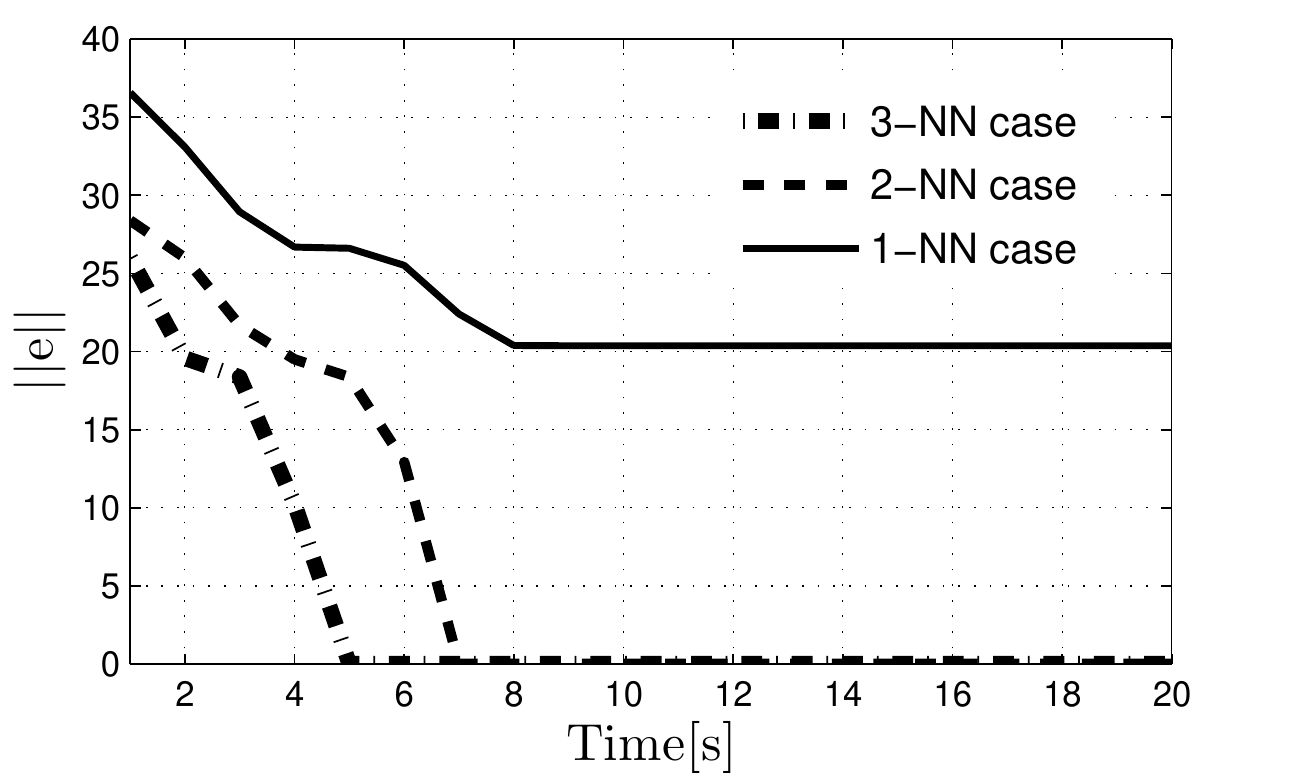}
\caption{Distributed calculation error for a vehicle in an 1, 2 and 3 nearest-neighbour platoons, with a single fault in vehicle 3.}
\label{fig:knn4}
\end{figure}
%-----------------------------------------------------
\begin{remark}
It should be noted that in distributed calculation algorithms mentioned in this paper, the magnitude of the calculation error $||e||$ is a decreasing function of time steps. This introduces an important feature of this algorithm, which says that even if vehicle $v_i$ is not willing to wait until the calculation error becomes zero and wants to pick the calculated values at some time step (before the error settles down to zero), it is sure that those values are closer to the real parameters than the initial guesses. This is the main advantage of this method, compared to network flooding algorithms \cite{Intanagonwiwat, Rahman}, in which vehicle $v_i$ must wait until the true value from other vehicles arrive. 
\end{remark}

\subsection{Velocity Fault Diagnostics}

After performing the distributed calculation algorithm, vehicle $v_i$ has enough information from all other vehicles in the network to examine the correctness of it own speed. More precisely, suppose that vehicle $v_i$ wants to find if there exists a failure in the velocity measurement (estimation) of itself or (possibly) any other vehicle in the network. Based on the algorithm mentioned in Section \ref{sec:fdetcor}, vehicle $v_i$ adopts an strategy  to calculate a residual function $\bar{e}_i^j$ for all $v_j\in \mathcal{V}\setminus \{v_i\}$. If $\bar{e}_i^j$ is nonzero (or above a certain threshold) for only a specific vehicle $v_j \in \mathcal{V}$, then the velocity $u_j$ is faulty. If $\bar{e}_i^j$ is nonzero for (almost) all vehicles $v_j \in \mathcal{V}$, then the velocity $u_i$ is faulty.  Fig. \ref{fig:kndsfasen2} shows that the residual signal, computed by vehicle $v_1$ in $\mathcal{P}(8,2)$, for vehicle $v_3$ (dashed line) is significantly larger than those of the rest of vehicles in the network. Hence, the velocity of vehicle $v_3$, which is $u_3$, is faulty. In this case, $v_3$ should calculate the opinions of other seven vehicles in the network about its own speed, i.e., $u_3^j$ for all $v_j\in \mathcal{V}$. Vehicle $v_3$ has the required materials for doing this calculation, which are all $p_i, u_i$ and $a_i$ for all other vehicles,  as it has gathered from the distributed calculation algorithm. After calculating all opinions $u_3^j$, $v_3$ does an averaging between the opinions and replaces its current velocity $u_3$ with the resulting velocity $\bar{u}_3$.

%-----------------------------------------------------
\begin{figure}[h!]
\centering
\includegraphics[width=0.95\linewidth]{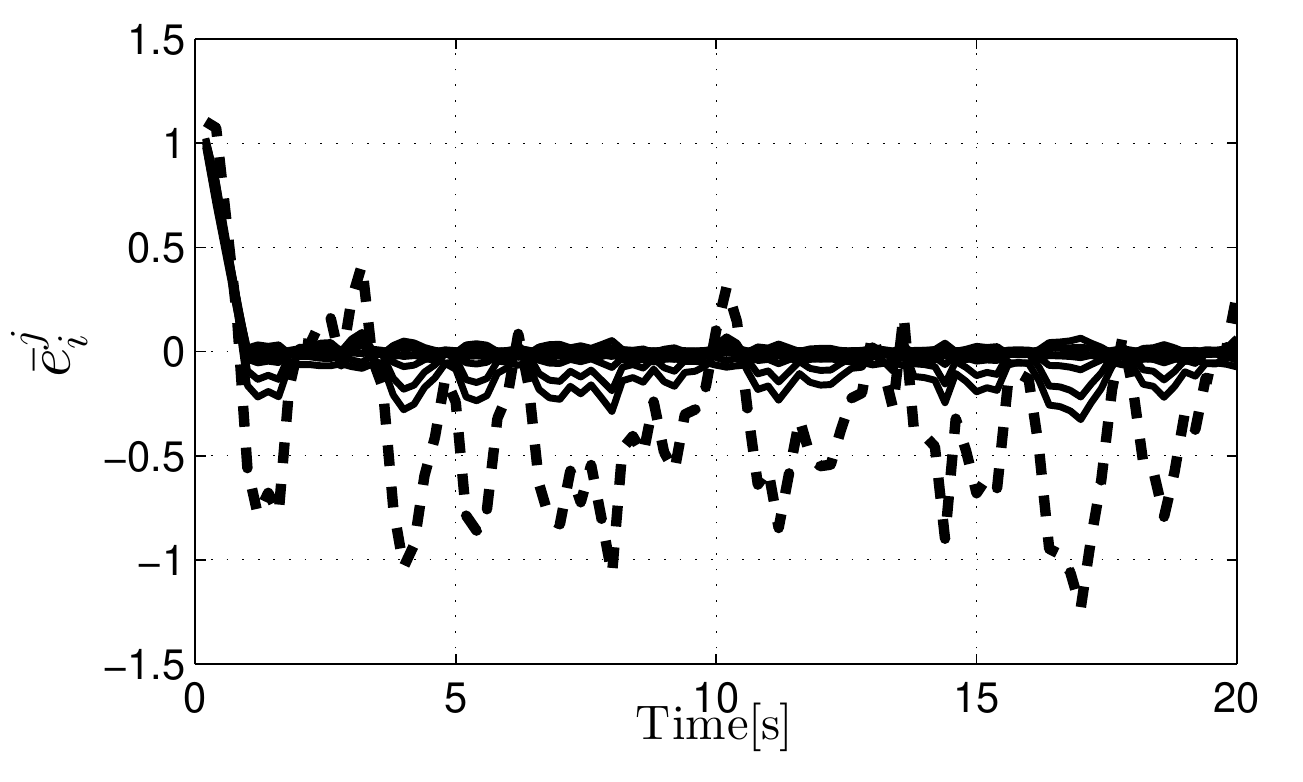}
\caption{The residual signals with respect to time calculated by vehicle $v_1$.}
\label{fig:kndsfasen2}
\end{figure}
%----------------------------------------------------- 
%%%%%%%%%%%%%%%%%%%%%%%%%%%%%%%%%%%%%%%%%%%%%%%%%%%%%%
\section{Summary and Conclusions} \label{sec:Conclusion}
%%%%%%%%%%%%%%%%%%%%%%%%%%%%%%%%%%%%%%%%%%%%%%%%%%%%%%
This paper introduced a cooperative vehicle longitudinal velocity (speed) fault detection and correction algorithm. To perform fault detection and correction algorithm, each vehicle requires to gather some information from other vehicles in the network, including speed, position and acceleration. Hence, a distributed function calculation strategy is used for each vehicle to gather these information from the network in a distributed manner. Then every vehicle operates a specific fault diagnosis algorithm to find out if there exists a failure in its own velocity estimation (or measurement) and correct it. The application of the distributed function calculation to vehicle networks as well as fault diagnosis and correction algorithms were the contributions of this paper. Several simulation results presented to validate the theoretical results. Experimental validation of these algorithms is an avenue for future studies.

\bibliographystyle{IEEEtran}
\bibliography{Main}

\end{document}